\documentclass[12pt, reqno]{amsart}
\usepackage{amsmath, amsthm, amscd, amsfonts, amssymb, graphicx, color}
\usepackage[bookmarksnumbered, colorlinks, plainpages]{hyperref}

\input{mathrsfs.sty}

\newtheorem{theorem}{Theorem}[section]
\newtheorem{lemma}[theorem]{Lemma}
\newtheorem{proposition}[theorem]{Proposition}
\newtheorem{corollary}[theorem]{Corollary}
\theoremstyle{definition}

\theoremstyle{remark}
\newtheorem{remark}[theorem]{Remark}
\numberwithin{equation}{section}
\begin{document}

\title [Some generalizations of numerical radius on off-diagonal part]{Some generalizations of numerical radius on off-diagonal part of $2\times 2$ operator matrices }

\author[M. Hajmohamadi,  R. Lashkaripour, M. Bakherad,   ]{ Monire Hajmohamadi$^1$, Rahmatollah Lashkaripour$^2$ and Mojtaba Bakherad$^3$}

\address{$^1$$^{,2}$$^{,3}$ Department of Mathematics, Faculty of Mathematics, University of Sistan and Baluchestan, Zahedan, I.R.Iran.}

\email{$^{1}$monire.hajmohamadi@yahoo.com}
\email{$^2$lashkari@hamoon.usb.ac.ir}

\email{$^{3}$mojtaba.bakherad@yahoo.com; bakherad@member.ams.org}


\begin{abstract}We generalize several inequalities involving powers of the numerical radius for off-diagonal part of $2\times2$ operator matrices of the form $T=\left[\begin{array}{cc}
 0&B\\
 C&0
 \end{array}\right]$, where $B, C$ are two operators.
 In particular,  if  $T=\left[\begin{array}{cc}
 0&B\\
 C&0
 \end{array}\right]$, then we get
\begin{align*}
 {1\over 2^{{3\over2}(r-1)}}\max\{ \| \mu \|, \| \eta \| \}
 \leq w^{r}(T)\leq \frac{1}{2^{r+1}}
 \max\{ \| \mu \|, \| \eta \| \},
 \end{align*}
 where $r\geq 2$ and
 $ \mu=|(C-B^{*})+i(C+B^{*})|^{r}+|(B^{*}-C)+i(C+B^{*})|^{r}$,

 $ \eta=|(B-C^{*})+i(B+C^{*})|^{r}+|(C^{*}-B)+i(B+C^{*})|^{r}$.

\end{abstract} \maketitle
\section{Introduction}
Let $({\mathscr H}, \langle\, .\,,\, .\,\rangle)$ be a complex Hilbert space and ${\mathbb B}(\mathscr H)$ denotes the $C^{*}$-algebra of all bounded linear operators on ${\mathscr H}$. In the case when ${\rm dim}{\mathscr
H}=n$, we identify ${\mathbb B}({\mathscr H})$ with the matrix
algebra $\mathbb{M}_n$ of all $n\times n$ matrices with entries in
the complex field.  The numerical radius of $T\in {\mathbb B}({\mathscr H})$ is defined by
\begin{align*}
w(T):=\sup\{\mid \langle Tx, x\rangle\mid : x\in {\mathscr H}, \parallel x \parallel=1\}.
\end{align*}
It is well known that $w(\,\cdot\,)$ defines a norm on ${\mathbb B}({\mathscr H})$, which is equivalent to the usual operator norm $\| \,.\, \|$. In fact, for any $T\in {\mathbb B}({\mathscr H})$,
$
\frac{1}{2}\| T \|\leq w(T) \leq\| T \|$;
 see \cite{gof}.
 An important inequality for $w(A)$ is the power inequality stating that $w(A^n)\leq w(A)^n\,\,(n=1,2,\cdots)$.
 It has been shown in \cite{naj}, that if $T\in {\mathbb B}({\mathscr H})$, then
\begin{align}\label{13}
w(T)\leq \frac{1}{2} \| |T| + |T^{*}| \|,
\end{align}
where $|T|=(T^{*}T)^{\frac{1}{2}}$ is the absolute value of $T$.
Recently in \cite{eli} the authors showed
 \begin{align}\label{7}
 w^{2r}(T)\leq  \frac{1}{2}\left(\|A\|^{2r}+\left\|\frac{1}{p} f^{pr}(\mid A^2 \mid)+\frac{1}{q} g^{qr}(\mid (A^{*})^2\mid) \right\|\right),
   \end{align}
  in which   $f$, $g$ are nonnegative  continuous  functions on $[0, \infty)$  satisfying the relation $f(t)g(t)=t\,(t\in [0, \infty))$,
  $r\geq 1$, $p \geq q >1$ such that $\frac{1}{p}+\frac{1}{q}=1$ and $pr \geq 2$.\\

 Let ${\mathscr H_{1}},{\mathscr H_{2}}, \cdots,{\mathscr H_{n}} $ be Hilbert spaces, and consider the direct sum ${\mathscr H}=\bigoplus_{j=1}^{n}{\mathscr H_{j}}$. With respect to this decomposition, every operator $T\in {\mathbb B}({\mathscr H})$ has an $n\times n$ operator matrix representation $T=[T_{ij}]$ with entries  $T_{ij}\in {\mathbb B}({\mathscr H_{j}}, {\mathscr H_{i}})$, the space of all bounded linear operators from ${\mathscr H_{j}}$ to ${\mathscr H_{i}}$.
Operator matrices provide a usual tool for studying Hilbert space operators, which have been extensively studied in the literatures.
 The classical Young inequality says that if $p, q>1$ such that $\frac{1}{p}+\frac{1}{q}=1$, then $ab\leq \frac{a^{p}}{p}+\frac{b^{q}}{q}$ for positive real numbers $a, b$.
A refinement of the scalar Young inequality is presented in \cite{FUJ} as following $(a^{\frac{1}{p}}b^{\frac{1}{q}})^{m}+r_{0}^{m}(a^{\frac{m}{2}}-b^{\frac{m}{2}})^{2}\leq(\frac{a}{p}+\frac{b}{q})^{m},$
where $r_{0}=\min \{ \frac{1}{p}, \frac{1}{q}\}$ and $m=1, 2,\cdots$. In particular, if $p=q=2$, then
\begin{align}\label{12}
(a^{\frac{1}{2}}b^{\frac{1}{2}})^{m}+(\frac{1}{2})^{m}(a^{\frac{m}{2}}-b^{\frac{m}{2}})^{2}\leq 2^{-m}(a+b)^{m}.
\end{align}
Let $T_{1}, T_{2},\cdots,T_{n}\in {\mathbb B}({\mathscr H})$.
The functional $w_{p}$ of operators $T_{1},\cdots,T_{n}$ for $p\geq 1$ is defined in \cite{FUJ2} as following
\begin{align*}
w_{p}(T_{1},\cdots,T_{n}):= \sup_{\| x \| = 1} (\sum_{i=1}^{n} | \langle T_{i}x, x\rangle |^{p})^{\frac{1}{p}}.
\end{align*}
In \cite{sheikh} the authors showed the following inequality
\begin{align*}
 w_{p}^p(A_1^*T_{1}B_1,\cdots,A_n^*T_{n}B_n)\leq {1\over2}\left\|\sum_{i=1}^n\left([B_i^*f^2(|T_i|)B_i]^p+[A_i^*g^2(
 |T_i^*|)A_i]^p\right)\right\|-\inf_{\|X\|=1} \zeta (X),
 \end{align*}
 where $A_i, B_i, T_i \in {\mathbb B}({\mathscr H})\,\,(i=1,2,\cdots,n)$, $f$, $g$ are nonnegative  continuous  functions on $[0, \infty)$ such that $f(t)g(t)=t\,(t\in [0, \infty))$, $p, r\geq m$, $m=1,2,\cdots,$   and
 {\footnotesize\begin{align*}
 \zeta (X)=2^{-m}\sum_{i=1}^n\left(\langle[B_i^*f^2(|T_i|)B_i]^{p\over m}x,x\rangle^{m \over2}-\langle[A_i^*g^2(
 |T_i^*|)A_i]^{p\over m}x,x\rangle^{m\over2}\right)^2.
 \end{align*}}
For further information about numerical radius inequalities we refer the
reader to \cite{aA, ando, sheikh} and references therein.

In this paper, we establish some generalizations of inequalities that is based on the off-diagonal parts of $2\times2$ operator matrices. We also show some inequalities involving powers of the numerical radius for the off-diagonal parts of $2\times 2$ operator matrices.


\section{main results}
To prove our first result, we need several well known lemmas.\\

\begin{lemma}\cite{ROD, YAM}\label{1}
Let $A\in {\mathbb B}({\mathscr H_1})$, $B\in {\mathbb B}({\mathscr H_2, \mathscr H_1})$, $C\in {\mathbb B}({\mathscr H_1,\mathscr H_2})$ and $D\in {\mathbb B}({\mathscr H_2})$. Then the following statements hold:

 $(a)\,\,w\left(\left[\begin{array}{cc}
 A&0\\
 0&D
 \end{array}\right]\right)$
 = $ \max \{{w(A), w(D)}\};$
\\

 $(b)\,\,w\left(\left[\begin{array}{cc}
 0&B\\
 C&0
 \end{array}\right]\right)$
 = $w\left(\left[\begin{array}{cc}
 0&C\\
 B&0
 \end{array}\right]\right);$
\\

 $(c)\,\,w\left(\left[\begin{array}{cc}
 0&B\\
 C&0
 \end{array}\right]\right)$
 = $\frac{1}{2} \sup_{\theta\in \mathcal{R}}\parallel e^{i\theta}B+e^{-i\theta}C^{*}\parallel; $
\\

 $(d)\,\,w\left(\left[\begin{array}{cc}
 A&B\\
 B&A
 \end{array}\right]\right)$
 = $\max \{w(A+B), w(A-B)\}.$
\\

In particular,

\begin{align*}
      w\left(\left[\begin{array}{cc}
              0&B\\
              B&0
              \end{array}\right]\right)
              = w(B).
\end{align*}
\end{lemma}
The second lemma is a simple consequence of the classical Jensen and Young inequalities; see \cite{ABM}.

\begin{lemma}\label{2}
Let $a, b\geq 0$ and $ p, q>1$ such that $\frac{1}{p} +\frac{1}{q}=1$. Then
\begin{align*}
ab\leq \frac {a^{p}}{p}+\frac {b^{q}}{q} \leq (\frac{a^{pr}}{p}+\frac{b^{qr}}{q})^{\frac{1}{r}}
\end{align*}
for $r \geq 1$.
\end{lemma}
The next lemma follows from the spectral theorem for positive operators and Jensen inequality; see \cite{KIT}.
\begin{lemma}\label{3}
(McCarty inequality). Let $T\in{\mathbb B}({\mathscr H})$, $ T \geq 0$ and $x\in {\mathscr H}$ be a unit vector. Then\\
$(a)\,\, \langle Tx, x\rangle^{r} \leq  \langle T^{r}x, x\rangle$ for $ r\geq 1;$\\
$(b)\,\,\langle T ^{r}x, x\rangle  \leq  \langle Tx, x\rangle^{r}$ for $ 0<r\leq 1$.\\
\end{lemma}
The following lemma is a consequence of convexity of the absolute value function.
\begin{lemma}\label{4}
Let $T\in {\mathbb B}({\mathscr H})$ be self-adjoint and $x\in {\mathscr H}$ be a unit  vector. Then
\begin{align*}
\mid \langle Tx, x\rangle \mid \leq \langle \mid T\mid x, x\rangle.
 \end{align*}
\end{lemma}
\begin{lemma}\cite[Theorem 1]{KIT}\label{5}
Let $T\in{\mathbb B}({\mathscr H})$ and $x, y\in {\mathscr H}$ be any vectors.\\
 $(a)$ If $f$, $g$ are nonnegative  continuous functions on $[0, \infty)$ which are satisfying the relation $f(t)g(t)=t\,(t\in[0, \infty))$, then
\begin{align*}
\mid \langle Tx, y \rangle \mid \leq \parallel f(\mid T \mid)x \parallel \parallel g(\mid T^{*} \mid)x \parallel;
 \end{align*}
 $(b)$ If $0\leq \alpha \leq 1$, then
 \begin{align*}
 \mid \langle Tx, y\rangle\mid^{2}\leq \langle \mid T\mid^{2\alpha}x, x\rangle\langle\mid T^{*}\mid^{2(1-\alpha)}y, y\rangle.
 \end{align*}
\end{lemma}
Now we are in a position to state the main results of this section.\\
\begin{theorem}
Let
$T=\left[\begin{array}{cc}
 0&B\\
 C&0
 \end{array}\right]\in {\mathbb B}({\mathscr H_2,\mathscr H_1})$ and  $f$, $g$ be nonnegative  continuous  functions on $[0, \infty)$ satisfying the relation $f(t)g(t)=t\,(t\in[0, \infty))$. Then
 \begin{align}\label{7}
 w^{r}(T)\leq \max \left\{ \left\| \frac{1}{p} f^{pr} (\mid C \mid) + \frac{1}{q} g^{qr} (\mid B^{*} \mid) \right\|, \left\|\frac{1}{p} f^{pr}(\mid B \mid)+\frac{1}{q} g^{qr}(\mid C^{*}\mid) \right\|\right\},
   \end{align}
 in which  $r\geq 1$, $p \geq q >1$ such that $\frac{1}{p}+\frac{1}{q}=1$ and $pr \geq 2$.\\
\end{theorem}
\begin {proof}
For any unit vector $X=\left[\begin{array}{cc}
 x_1\\
 x_2
 \end{array}\right] \in \mathscr {H}_1\oplus \mathscr {H}_2$ we have
 \begin{align*}
 \mid \langle &TX, X \rangle \mid^{r}\\
 &\leq \parallel f\left(\mid T \mid\right)X\parallel^{r} \parallel g\left(\mid T^{*}\mid \right)X\parallel ^{r}    \qquad (\textrm {by Lemma\,\,}\ref{5})
  \\&= \langle f^{2}(\mid T \mid)X, X\rangle^{\frac{r}{2}} \langle g^{2}(\mid T^{*} \mid)X, X\rangle^{\frac{r}{2}}
 \\&
 \leq \frac{1}{p}\left\langle f^{2}\left(
 \left[\begin{array}{cc}
 \mid C \mid & 0\\
 0 & \mid B \mid
 \end{array}\right]
 \right)
 X, X\right\rangle ^{\frac{pr}{2}}+\frac{1}{q} \left\langle g^{2}\left(
 \left[\begin{array}{cc}
 \mid B^{*} \mid & 0\\
 0 & \mid C^{*} \mid
 \end{array}\right]
 \right)
 X, X\right\rangle ^{\frac{qr}{2}}  \\&     \qquad \qquad \qquad     \qquad (\textrm { by Lemma }\ref{2})
 \\&
 \leq \frac{1}{p}\left\langle \left[\begin{array}{cc}
 f^{pr} \mid C \mid & 0\\
 0 & f^{pr} \mid B \mid
 \end{array}\right]
 X, X\right\rangle +\frac{1}{q} \left\langle \left[\begin{array}{cc}
 g^{qr} \mid B^{*} \mid & 0\\
 0 & g^{qr} \mid C^{*} \mid
 \end{array}\right]
 X, X\right\rangle\\&\qquad   \qquad \qquad \qquad             (\textrm {by Lemma}\, \ref{3}\textrm{(a)})
 \\&
 =\left\langle \left[\begin{array}{cc}
 \frac{1}{p} f^{pr} (\mid C \mid) +\frac{1}{q} g^{qr} (\mid B^{*}\mid) & 0\\
 0 & \frac{1}{p}f^{pr} (\mid B \mid)+\frac{1}{q} g^{qr}(\mid C^{*} \mid)
 \end{array}\right]
 X, X\right\rangle.
  \end{align*}
 Then
 \begin{align*}
 \mid \langle TX, X \rangle\mid^{r} \leq \left\langle \left[\begin{array}{cc}
 \frac{1}{p} f^{pr} (\mid C \mid) +\frac{1}{q} g^{qr} (\mid B^{*}\mid) & 0\\
 0 & \frac{1}{p}f^{pr} (\mid B \mid)+\frac{1}{q} g^{qr}(\mid C^{*} \mid)
 \end{array}\right]
 X, X\right\rangle .
 \end{align*}
 Now, applying the definition of numerical radius and Lemma \ref{1}(a),  we have
 \begin{align*}
 w^{r}(T)\leq \max \left\{ \left\| \frac{1}{p} f^{pr} (\mid C \mid) + \frac{1}{q} g^{qr} (\mid B^{*} \mid) \right\|, \left\|\frac{1}{p} f^{pr}(\mid B \mid)+\frac{1}{q} g^{qr}(\mid C^{*}\mid) \right\|\right\}.
 \end{align*}
\end{proof}
\begin{corollary}\cite[Corollary 3]{CAL}\label{100}
Let
$T=\left[\begin{array}{cc}
 0&B\\
 C&0
 \end{array}\right]\in {\mathbb B}({\mathscr H_2,\mathscr H_1})$
 be a positive  operator matrix and $r\geq 1$. Then
 \begin{align*}
 w(T)=\frac{1}{2}\parallel B+C \parallel.
 \end{align*}
 \end{corollary}
 \begin{proof}
 Putting $f(t)=g(t)=t^{\frac{1}{2}}, r=1$ and $p=q=2$ in inequality \eqref{7} and applying Lemma \ref{1}(c), we get the equality.
 \end{proof}

  \begin{theorem}
Let
$T=\left[\begin{array}{cc}
 0&B\\
 C&0
 \end{array}\right]\in {\mathbb B}({\mathscr H_2,\mathscr H_1})$ and  $f$, $g$ be nonnegative  continuous  functions on $[0, \infty)$ satisfying the relation $f(t)g(t)=t\,\,(t\in [0, \infty))$. Then
 \begin{align}\label{6}
 w^{2r}(T)\leq \max \left\{\left\| \frac{1}{p} f^{2pr}(\mid C \mid)+\frac{1}{q} g^{2qr}(\mid B^{*}\mid)\right\| , \left\| \frac{1}{p} f^{2pr}(\mid B \mid)+\frac{1}{q} g^{2qr}(\mid C^{*}\mid)\right\| \right\},
   \end{align}
 where $r\geq 1$ and $p \geq q >1$ such that $\frac{1}{p}+\frac{1}{q}=1$ and $pr \geq 1$.\\
\end{theorem}
\begin {proof}
Assume that $X=\left[\begin{array}{cc}
 x_1\\
 x_2
 \end{array}\right] \in \mathscr {H}_1\oplus \mathscr {H}_2$ is a unit vector. Then
 \begin{align*}
 \mid \langle &TX, X \rangle \mid^{2r}
  \\& \leq \parallel f(\mid T \mid)X\parallel^{2r} \parallel g(\mid T^{*}\mid X\parallel ^{2r}       \qquad (\textrm {by Lemma}\,\, \ref{5})
  \\&
 = \langle f^{2}(\mid T \mid)X, X\rangle^{r} \langle g^{2}(\mid T^{*} \mid)X, X\rangle^{r}
 \\
 & \leq \frac{1}{p}\left\langle f^{2}\left(
 \left[\begin{array}{cc}
 \mid C \mid & 0\\
 0 & \mid B \mid
 \end{array}\right]
\right )
 X, X\right\rangle ^{rp}+\frac{1}{q} \left\langle g^{2}\left(
 \left[\begin{array}{cc}
 \mid B^{*} \mid & 0\\
 0 & \mid C^{*} \mid
 \end{array}\right]
 \right)
 X, X\right\rangle ^{rq}\\&\qquad \qquad \qquad            \qquad (\textrm {by Lemma} \,\,\ref{2})
 \\&
  \leq \frac{1}{p}\left\langle \left[\begin{array}{cc}
 f^{2pr} \mid C \mid & 0\\
 0 & f^{2pr} \mid B \mid
 \end{array}\right]
 X, X\right\rangle +\frac{1}{q} \left\langle \left[\begin{array}{cc}
 g^{2qr} \mid B^{*} \mid & 0\\
 0 & g^{2qr} \mid C^{*} \mid
 \end{array}\right]
 X, X\right\rangle\\&\qquad \qquad \qquad    \qquad (\textrm {by Lemma}\,\,\ref{3}\textrm{(a)})
 \\&
 =\left\langle \left[\begin{array}{cc}
 \frac{1}{p} f^{2pr} (\mid C \mid) +\frac{1}{q} g^{2qr} (\mid B^{*}\mid) & 0\\
 0 & \frac{1}{p}f^{2pr} (\mid B \mid)+\frac{1}{q} g^{2qr}(\mid C^{*} \mid)
 \end{array}\right]
 X, X\right\rangle.
 \end{align*}
 Thus
  \begin{align*}
 \mid \langle TX, X\rangle \mid^{2r}\leq \left\langle \left[\begin{array}{cc}
 \frac{1}{p} f^{2pr} (\mid C \mid) +\frac{1}{q} g^{2qr} (\mid B^{*}\mid) & 0\\
 0 & \frac{1}{p}f^{2pr} (\mid B \mid)+\frac{1}{q} g^{2qr}(\mid C^{*} \mid)
 \end{array}\right]
 X, X\right\rangle.
 \end{align*}
  Now by the definition of numerical radius and Lemma \ref{1}(a), we have
 \begin{align*}
 w^{2r}(T)\leq \max \left\{\left\| \frac{1}{p} f^{2pr}(\mid C \mid)+\frac{1}{q} g^{2qr}(\mid B^{*}\mid)\right\| , \left\| \frac{1}{p} f^{2pr}(\mid B \mid)+\frac{1}{q} g^{2qr}(\mid C^{*}\mid)\right\| \right\}.
 \end{align*}
\end{proof}
Inequality \eqref{6} induces several numerical radius inequalities as follows.
\begin{corollary}
Let
$T=\left[\begin{array}{cc}
 0&B\\
 C&0
 \end{array}\right]\in {\mathbb B}({\mathscr H_2,\mathscr H_1})$. Then
\begin{align*}
 w^{2r}(T)\leq \frac{1}{2}\max\{\parallel \mid C \mid^{4r\alpha}+\mid B^{*} \mid^{4r(1-\alpha)} \parallel, \parallel \mid B \mid^{4r\alpha}+\mid C^{*} \mid^{4r(1-\alpha)} \parallel \}
 \end{align*}
 for any $r\geq 1$ and $0\leq \alpha\leq1$.
\end{corollary}
\begin{proof}
Letting $f(t)=t^{\alpha}, g(t)=t^{1-\alpha}$ and $p=q=2$ in inequality \eqref{6}, we get the desired inequality.
\end{proof}
\begin{corollary}\label{101}
Let $B\in {\mathbb B}(\mathscr H)$,  $0\leq \alpha \leq 1$ and $r\geq1$. Then
\begin{align}\label{pow}
 w^{2r}(B) \leq \frac{1}{2} \parallel \mid B \mid^{4r\alpha} + \mid B^{*} \mid^{4r(1-\alpha)}\parallel.
 \end{align}
 \end{corollary}
\begin{proof}
We put $f(t)=t^{\alpha}, g(t)=t^{1-\alpha}, p=q=2$ and
 $T=\left[\begin{array}{cc}
 0&B\\
 B&0
 \end{array}\right]$ and apply Lemma \ref{1}(d),  we get the desired result.
\end{proof}

 \begin{theorem}
 Let
 $T_{i}=\left[\begin{array}{cc}
 0&B_{i}\\
 C_{i}&0
 \end{array}\right]
 \in {\mathbb B}({\mathscr H_2\oplus\mathscr H_1})$ for any $i=1, 2,\cdots,n$. Then
 {\footnotesize\begin{align}
 w_{p}^{p}(T_{1}, T_{2},\cdots,T_{n})\leq \max\left\{\left\| \sum_{i=1}^{n} \alpha \mid C_{i}\mid^{p}+(1-\alpha)\mid B_{i}^{*}\mid^{p}\right\|, \left\| \sum_{i=1}^{n} \alpha \mid B_{i}\mid^{p}+(1-\alpha)\mid C_{i}^{*}\mid^{p}\right\|\right\}
 \end{align}}
 for $0\leq \alpha\leq 1 $ and $p\geq 2$.
 \end{theorem}
 \begin{proof}
 For any unit vector $X=\left[\begin{array}{cc}
 x_1\\
 x_2
 \end{array}\right]
 \in {\mathscr H_1\oplus\mathscr H_2}$, we have\\
 {\footnotesize\begin{align*}
 \sum_{i=1}^{n}\mid\langle T_{i}X, X\rangle\mid^{p}
 &=\sum_{i=1}^{n}(\mid\langle T_{i}X, X\rangle\mid^{2})^{\frac{p}{2}}\\&
 \leq \sum_{i=1}^{n}(\langle\mid T_{i}\mid^{2\alpha}X, X\rangle \langle \mid T_{i}^{*}\mid^{2(1-\alpha)}X, X\rangle)^{\frac{p}{2}}   \qquad (\textrm {by Lemma}\, \ref{5} \,(b))\\&
 \leq \sum_{i=1}^{n}\langle \mid T_{i}\mid^{p\alpha}X, X\rangle\langle \mid T_{i}^{*}\mid^{p(1-\alpha)}X, X\rangle
 \qquad  \qquad (\textrm {by Lemma}\, \ref{3}\,(b))     \\&
 \leq \sum_{i=1}^{n}\langle \mid T_{i}\mid^{p}X, X\rangle^{\alpha}\langle\mid T_{i}^{*}\mid^{p}X, X\rangle^{1-\alpha}\\&
 \leq\sum_{i=1}^{n}(\alpha \langle\mid T_{i}\mid^{p}X, X\rangle+(1-\alpha)\langle \mid T_{i}^{*}\mid^{p}X, X\rangle)
 \qquad (\textrm {by Lemma}\, \ref{2})  \\&
 =\sum_{i=1}^{n}\left(
 \alpha\left\langle
 \left[\begin{array}{cc}
 \mid C_{i}\mid^{p}&0\\
 0&\mid B_{i}\mid^{p}
 \end{array}\right]
 X, X\right\rangle + (1-\alpha)\left\langle
 \left[\begin{array}{cc}
 \mid B_{i}^{*}\mid^{p}&0\\
 0&\mid C_{i}^{*}\mid^{p}
 \end{array}\right]
 X, X\right\rangle
\right)
\\&
=\sum_{i=1}^{n} \left\langle
\left[\begin{array}{cc}
 \alpha\mid C_{i}\mid^{p}+(1-\alpha)\mid B_{i}^{*}\mid^{p}&0\\
 0&\alpha \mid B_{i}\mid^{p}+(1-\alpha)\mid C_{i}^{*}\mid^{p}
 \end{array}\right]
 X, X\right\rangle\\&
 =\left\langle
 \left[\begin{array}{cc}
 \sum_{i=1}^{n}\alpha\mid C_{i}\mid^{p}+(1-\alpha)\mid B_{i}^{*}\mid^{p}&0\\
 0& \sum_{i=1}^{n}\alpha \mid B_{i}\mid^{p}+(1-\alpha)\mid C_{i}^{*}\mid^{p}
 \end{array}\right]
 X, X\right\rangle.
 \end{align*}}
 By the definition of numerical radius and Lemma \ref{1}, we have\\
 {\footnotesize\begin{align*}
 w_{p}^{p}(T_{1}, T_{2},\cdots,T_{n})\leq \max\left\{\left\| \sum_{i=1}^{n} \alpha \mid C_{i}\mid^{p}+(1-\alpha)\mid B_{i}^{*}\mid^{p}\right\|, \left\| \sum_{i=1}^{n} \alpha \mid B_{i}\mid^{p}+(1-\alpha)\mid C_{i}^{*}\mid^{p}\right\|\right\}.
 \end{align*}}
 \end{proof}
 \begin{remark}
 As a special case for $\alpha=\frac{1}{2}$ and $B_{i}=C_{i}$ for any $i=1, 2,\cdots,n$, we have the following inequality \\
 \begin{align*}
 w_{p}^{p}(B_{1},B_{2},\cdots,B_{n})\leq \frac{1}{2}\parallel \sum_{i=1}^{n}\mid B_{i}\mid^{p}+\mid B_{i}^{*}\mid^{p}\parallel,
 \end{align*}
 which already shown in \cite[Proposition 3.9]{FUJ2}.
 \end{remark}
 Now using a refinement of the classical Young inequality, we have the following theorem.\\
\begin{theorem}
Let
$T=\left[\begin{array}{cc}
 0&B\\
 C&0
 \end{array}\right]\in {\mathbb B}({\mathscr H_2,\mathscr H_1})$ and  $f$, $g$ be nonnegative  continuous  functions on $[0, \infty)$  satisfying the relation $f(t)g(t)=t$ $(t\in [0, \infty))$. Then for $m=1,2,\cdots$ and $p, r\geq m$
\begin{align}\label{11}
 w^{r}(T)\leq (\frac{1}{2})^{m} \max\{\parallel f^{\frac{2r}{m}}\mid C \mid+g^{\frac{2r}{m}}\mid B^{*}\mid\parallel^{m}, \parallel f^{\frac{2r}{m}}\mid B \mid+g^{\frac{2r}{m}}\mid C^{*}\mid \parallel^{m} \}-\inf_{\|X\|=1} \zeta (X),
 \end{align}
 where
 {\footnotesize\begin{align*}
 \zeta (X)=2^{-m}\left(\left\langle f^{\frac{2r}{m}}
 \left[\begin{array}{cc}
 \mid C\mid&0\\
 0&\mid B \mid
 \end{array}\right]
 X, X\right\rangle^{\frac{m}{2}}-\left\langle g^{\frac{2r}{m}}
 \left[\begin{array}{cc}
 \mid B^{*}\mid&0\\
 0&\mid C^{*}\mid
 \end{array}\right]
 X, X\right\rangle^{\frac{m}{2}}\right)^{2}.
 \end{align*}}
 \end{theorem}
 \begin{proof}
 Let $X=\left[\begin{array}{cc}
 x_1\\
 x_2
 \end{array}\right]
 \in {\mathscr H_1\oplus\mathscr H_2}$ be  a unit vector. Applying Lemmas \ref{5}, \ref{3} and inequality \eqref{12}, respectively, we have
  {\footnotesize\begin{align*}
 \mid \langle TX, X\rangle\mid^{r}&\leq\parallel f(\mid T\mid)X \parallel^{r}\parallel g(\mid T^{*}\mid)X\parallel^{r}\\&
 =\left(\langle f^{2}(\mid T\mid)X, X\rangle^{\frac{r}{2m}}\langle g^{2}(\mid T^{*}\mid)X, X\rangle^{\frac{r}{2m}}\right)^{m}\\&
 \leq\left(\langle f^{\frac{2r}{m}}(\mid T\mid)X, X\rangle^{\frac{1}{2}}\langle g^{\frac{2r}{m}}(\mid T^{*}\mid)X, X\rangle^{\frac{1}{2}}\right)^{m}\\&
 -2^{-m}\left(\langle f^{\frac{2r}{m}}(\mid T\mid)X, X\rangle^{\frac{m}{2}}-\langle g^{\frac{2r}{m}}(\mid T^{*}\mid)X, X\rangle^{\frac{m}{2}}\right)^{2}\\&
 \leq\left(\frac{1}{2}\left\langle f^{\frac{2r}{m}}
\left[\begin{array}{cc}
 \mid C\mid&0\\
 0&\mid B \mid
 \end{array}\right]
 X, X\right\rangle+\frac{1}{2}\left\langle g^{\frac{2r}{m}}
 \left[\begin{array}{cc}
 \mid B^{*}\mid&0\\
 0&\mid C^{*} \mid
 \end{array}\right]
 X, X\right\rangle\right)^{m}\\&
 -2^{-m}\left(\left\langle f^{\frac{2r}{m}}
 \left[\begin{array}{cc}
 \mid C\mid&0\\
 0&\mid B \mid
 \end{array}\right]
 X, X\right\rangle^{\frac{m}{2}}-\left\langle g^{\frac{2r}{m}}
 \left[\begin{array}{cc}
 \mid B^{*}\mid&0\\
 0&\mid C^{*} \mid
 \end{array}\right]
 X, X\right\rangle^{\frac{m}{2}} \right)^{2}\\&
 =\left(\frac{1}{2}
 \left\langle
 \left[\begin{array}{cc}
 f^{\frac{2r}{m}}\mid C\mid +g^{\frac{2r}{m}}\mid B^{*}\mid&0\\
 0&f^{\frac{2r}{m}}\mid B \mid +g^{\frac{2r}{m}}\mid C^{*}\mid
 \end{array}\right]
 X, X\right\rangle \right)^{m}\\&
 -2^{-m}\left(\left\langle f^{\frac{2r}{m}}
 \left[\begin{array}{cc}
 \mid C\mid&0\\
 0&\mid B \mid
 \end{array}\right]
 X, X\right\rangle^{\frac{m}{2}}-\left\langle g^{\frac{2r}{m}}
 \left[\begin{array}{cc}
 \mid B^{*}\mid&0\\
 0&\mid C^{*} \mid
 \end{array}\right]
 X, X\right\rangle^{\frac{m}{2}} \right)^{2}.
 \end{align*}}
 Therefore
 {\footnotesize\begin{align*}
 w^{r}(T)&
 \leq (\frac{1}{2})^{m}\max\{\parallel f^{\frac{2r}{m}} \mid C\mid+g^{\frac{2r}{m}}\mid B^{*}\mid\parallel^{m}, \parallel f^{\frac{2r}{m}} \mid B\mid+g^{\frac{2r}{m}}\mid C^{*}\mid \parallel^{m}\}-\inf_{\|X\|=1} \zeta (X).
 \end{align*}}
 Hence we get the desired inequality.
 \end{proof}
 \begin{remark}
 In inequality \eqref{11} if $m=1,$ then we get a refinement of inequality \eqref{7}.
 \end{remark}
 \section{numerical radius of the operator matrix $2\times 2$}
 In this section, we estimate numerical radius of matrix
 $\left[\begin{array}{cc}
  A&B\\
 C&D
 \end{array}\right].$
 \begin{lemma}
 Let $ T=\left[\begin{array}{cc}
  A&0\\
 0&D
 \end{array}\right] \in {\mathbb B}({\mathscr H_{1}}\oplus {\mathscr H_{2}})$. Then
 \begin{align}\label{14}
 w^{r}(T)\leq \frac{1}{2}\max\{ \| |A|^{r}+|A^{*}|^{r}\|, \| |D|^{r}+|D^{*}|^{r}\|\}
 \end{align}
 for $r\geq 1$.
 \end{lemma}
 \begin{proof}
 Let  $X=\left[\begin{array}{cc}
 x_1\\
 x_2
 \end{array}\right]
 \in {\mathscr H_1\oplus\mathscr H_2}$ be any unit vector. Then
 \begin{align*}
 |\langle TX, X\rangle|&
 \leq \langle |T|X, X\rangle^{\frac{1}{2}}\langle |T^{*}|X, X\rangle^{\frac{1}{2}}\\&
 \leq \frac{1}{2}\left\langle \left[\begin{array}{cc}
 \mid A\mid&0\\
 0&\mid D \mid
 \end{array}\right]X, X \right\rangle+\frac{1}{2}\left\langle \left[\begin{array}{cc}
 \mid A^{*}\mid&0\\
 0&\mid D^{*} \mid
 \end{array}\right]X, X \right\rangle\\&
 \leq \left(\frac{1}{2}\left\langle \left[\begin{array}{cc}
 \mid A\mid&0\\
 0&\mid D \mid
 \end{array}\right]X, X \right\rangle^{r}+\frac{1}{2}\left\langle \left[\begin{array}{cc}
 \mid A^{*}\mid&0\\
 0&\mid D^{*} \mid
 \end{array}\right]X, X \right\rangle^{r}\right)^{\frac{1}{r}}\\&
 \leq \left(\frac{1}{2}\left\langle \left[\begin{array}{cc}
 \mid A\mid^{r}&0\\
 0&\mid D \mid^{r}
 \end{array}\right]X, X \right\rangle+\frac{1}{2}\left\langle \left[\begin{array}{cc}
 \mid A^{*}\mid^{r}&0\\
 0&\mid D^{*} \mid^{r}
 \end{array}\right]X, X \right\rangle\right)^{\frac{1}{r}}\\&
 =\left(
 \left\langle \left[\begin{array}{cc}
 \frac{1}{2}(\mid A\mid^{r}+|A^{*}|^{r})&0\\
 0&\frac{1}{2}(\mid D \mid^{r}+|D^{*}|^{r})
 \end{array}\right]X, X \right\rangle\right)^{\frac{1}{r}},
 \end{align*}
 and so \begin{align*}
 |\langle TX, X\rangle |^{r}\leq \left\langle \left[\begin{array}{cc}
 \frac{1}{2}(\mid A\mid^{r}+|A^{*}|^{r})&0\\
 0&\frac{1}{2}(\mid D \mid^{r}+|D^{*}|^{r})
 \end{array}\right]X, X \right\rangle.
 \end{align*}
 Therefore
 \begin{align*}
 w^{r}(T)\leq \frac{1}{2}\max\{ \| |A|^{r}+|A^{*}|^{r}\|, \| |D|^{r}+|D^{*}|^{r}\|\}.
 \end{align*}
 \end{proof}
 \begin{remark}
 By letting $r=1$ and $A=D$ in inequality \eqref{14}, we obtain inequality \eqref{13}, that is
 \begin{align*}
 w(A)\leq \frac{1}{2}\| |A|+|A^{*}| \|.
 \end{align*}
 \end{remark}
 The following proposition follows from inequalities \eqref{7} and \eqref{14}.
 \begin{proposition}
 Let  $T=\left[\begin{array}{cc}
  A&B\\
 C&D
 \end{array}\right]$ with $A, B, C, D\in {\mathbb B}({\mathscr H})$. Then
 \begin{align*}
 w(T)\leq \frac{1}{2}\max\{ \| |C| +| B^{*}|\|, \| | B| + |C^{*}|\|\}+\frac{1}{2}\max\{ \||A| +|A^{*}|\|, \| |D|+ |D^{*}|\| \}.
 \end{align*}
 In particular,
 \begin{align*}
  w\left(\left[\begin{array}{cc}
  A&B\\
 B&A
 \end{array}\right]\right)
 \leq \frac{1}{2}( \| |A|+ |A^{*}|\|+\| |B|+|B^{*}|\|).\\
 \end{align*}
 \end{proposition}

  \begin{theorem}
 Let  $ T=\left[\begin{array}{cc}
  0&B\\
 C&0
 \end{array}\right] \in {\mathbb B}({\mathscr H_{2}}\oplus {\mathscr H_{1}})$  and $r\geq 2$. Then
 \begin{align}
{1\over 2^{{3\over2}(r-1)}}\max\{ \| \mu \|, \| \eta \| \}
 \leq w^{r}(T)\leq (\frac{1}{2})^{r+1}
 \max\{ \| \mu \|, \| \eta \| \},
 \end{align}
 where
\begin{align*}
  \mu=|(C-B^{*})+i(C+B^{*})|^{r}+|(B^{*}-C)+i(C+B^{*})|^{r},
\end{align*}
   and
 \begin{align*}
  \eta=|(B-C^{*})+i(B+C^{*})|^{r}+|(C^{*}-B)+i(B+C^{*})|^{r}.
 \end{align*}
 \end{theorem}
 \begin{proof}
 Let  $X=\left[\begin{array}{cc}
  x_{1}\\
 x_{2}
 \end{array}\right] \in {\mathscr H_{1}}\oplus {\mathscr H_{2}}$ be a unit vector. Let  $T=S+iW$ be the Cartesian decomposition of $T$. Then applying \cite[Theorem 1]{kit2}, we have
 \begin{align*}
 w^{2}(T)\geq \frac{1}{2} \|(S\pm W)^{2}\|.
 \end{align*}
 Therefore
 \begin{align*}
 w^{r}(T)\geq 2^{-{\frac{r}{2}}}\| (S\pm W)^{2}\|^{\frac{r}{2}}=2^{-\frac{r}{2}}\|  |S\pm W |^{r}\|,
 \end{align*}
 and so
 \begin{align*}
 2w^{r}(T)&
 \geq 2^{-\frac{r}{2}}(\| S+W |^{r}\|+\| |S-W|^{r}\|)\\&
 \geq 2^{-\frac{r}{2}}\| |S+W|^{r} +| S-W|^{r} \|\\&
 \geq 2^{-\frac{r}{2}-1}|\langle (|S+W|^{r}+|S-W|^{r})X, X\rangle |\\&
 =2^{-\frac{r}{2}-1}\left| \left\langle
 \left[\begin{array}{cc}
  (\frac{1}{2})^{r}\mu&0\\
 0&(\frac{1}{2})^{r}\eta
 \end{array}\right]
 X, X\right\rangle \right|,
 \end{align*}
 where
 \begin{align*}
 \mu=| (C-B^{*})+i(C+B^{*})|^{r}+|(B^{*}-C)+i(C+B^{*})|^{r},
 \end{align*}
 and
 \begin{align*}
 \eta= |(B-C^{*})+i(B+C^{*})|^{r}+|(C^{*}-B)+i(B+C^{*})|^{r}.
 \end{align*}
 Taking the supremum over $X\in {\mathbb B}({\mathscr H_{1}}\oplus {\mathscr H_{2}})$ with $\| X\|=1$ in the above inequality and applying the numerical radius of diagonal matrices, we deduce the first inequality.\\
 For the second inequality, we have
 \begin{align*}
 |\langle TX, X\rangle|^{r}&=(\langle SX, X\rangle^{2}+\langle WX, X\rangle^{2})^{\frac{r}{2}}\\
 &=2^{-\frac{r}{2}}(\langle (S+W)X, X\rangle^{2}+\langle (S-W)X, X\rangle^{2})^{\frac{r}{2}}\\
 &\leq 2^{-\frac{r}{2}}2^{\frac{r}{2}-1}(|\langle (S+W)X, X\rangle|^{r}+|\langle (S-W)X, X\rangle|^{r})\\&
 \qquad \qquad \qquad (\textrm {since}\,  f(t)=t^{\frac{r}{2}}\textrm { is convex)}\\
 &\leq \frac{1}{2}(\langle | S+W |X, X\rangle^{r}+\langle | S-W |X, X\rangle^{r})\\&
 \leq  \frac{1}{2}(\langle | S+W |^{r}X, X\rangle+\langle | S-W |^{r}X, X\rangle)\\&
 =\frac{1}{2}\langle (| S+W |^{r}+| S-W |^{r})X, X\rangle\\&
 =\frac{1}{2}\left\langle
 \left[\begin{array}{cc}
  (\frac{1}{2})^{r}\mu&0\\
 0&(\frac{1}{2})^{r}\eta
 \end{array}\right]
 X, X\right\rangle.
 \end{align*}
 Now, applying the definition of numerical radius and Lemma \ref{1}, we get the desired inequality.
 \end{proof}
 \begin{remark}
 If $T^{2}=0$, then $w(T)=\frac{1}{2}\| T \|$, $\| T^{*}T+TT^{*} \|=\| T \|^{2},$ and
 $$\| |S+W|^{r}+| S-W|^{r} \|=2^{-\frac{r}{2}+1}\| T^{*}T+TT^{*} \|^{\frac{r}{2}}=2^{-\frac{r}{2}+1}\| T \|^{r}.$$
 On the other hand, from $\| |S+W|^{r}+| S-W|^{r} \|=\sup_{\|X\|=1}| \langle |S+W|^{r}+| S-W|^{r}X, X \rangle |$,
 we conclude that $(\frac{1}{2})^{r}\max\{ \| \mu\|, \|\eta\|\}=2^{-\frac{r}{2}+1}\| T \|^{r}$.
 Therefore $2^{\frac{-3}{2}r-1}\max\{ \| \mu \|, \| \eta \| \}=2^{-r}\| T \|^{r}=w^{r}(T)$, where $\mu$ and $\eta$ are defined above.\\
 \\
 \end{remark}

\bigskip
\bibliographystyle{amsplain}

\begin{thebibliography}{99}

\bibitem{aA} A. Abu-Omar and F. Kittaneh, \textit{Estimates for the numerical radius and the spectral radius of the Frobenius companion matrix and bounds for the zeros of polynomials}, Ann. Func. Anal. \textbf{5} (2014), no. 1, 56--62.

\bibitem{CAL} A. Abu-Omar and F. Kittaneh, \textit{Numerical radius inequalities for $n\times n$ operator matrices}, Linear Algebra Appl. \textbf{468} (2015), 18--26.

\bibitem{FUJ} Y. Al-manasrah and F. Kittaneh, \textit{A generalization of two refined Young inequalities}, Positivity  \textbf{19} (2015), no. 4, 757--768.

\bibitem{ando}P.R. Halmos, \textit{A Hilbert Space Problem Book}, 2nd ed., springer, New York, 1982.

\bibitem{ABM} G.H. Hardy, J.E. Littlewood and G. Polya,  \textit{inequalities}, 2nd ed., Cambridge Univ. Press, Cambridge,  1988.

\bibitem{ROD} O. Hirzallah, F. Kittaneh and K. Shebrawi, \textit{Numerical radius inequalities for certain 2$\times$2 operator matrices}, Integral equations Operator Theory \textbf{71} (2011), 129-149.

\bibitem{KIT} F. Kittaneh, \textit{Notes on some inequalitis for Hilbert space operators}, Publ. Res. Inst. Math. Sci. \textbf{24} (2) (1988), 283--293.

\bibitem{naj} F. Kittaneh, \textit{A numerical radius inequality and an estimate for the numerical radius of the Frobenius companion matrix},
Studia Math. \textbf{158} (2003), 11--17.

\bibitem{kit2} F. Kittaneh, \textit{Numerical radius inequalities for Hilbert space operators},
Studia Math. \textbf{168} (2005), no. 1, 73--80.

\bibitem{MOS} F. Kittaneh, M.S. Moslehian and T. Yamazaki, \textit{Cartesian decomposition and numerical radius inequalities}, Linear Algebra Appl. \textbf{471} (2015), 46--53.

\bibitem{gof}K.E. Gustafson and D.K.M. Rao, \textit{Numerical Range, The Field of Values of Linear Operators and
Matrices}, Springer, New York, 1997.


\bibitem{eli} M. Sattari, M.S. Moslehian and T. Yamazaki, \textit{Some genaralized numerical radius inequalities for Hilbert space operators}, Linear Algebra Appl, \textbf{470} (2014), 1--12.

\bibitem{FUJ2}  M. Sattari, M.S. Moslehian and K. Shebrawi, \textit{Extension of Euclidean operator radius inequalities},  Math. Scand. (2016) in press, arXiv 1502.00083.

\bibitem{sheikh} A. Sheikhhosseini, M.S. Moslehian and K. Shebrawi, \textit{Inequalities for generalized Euclidean operator radius via Young's inequality}, J. Math. Anal. Appl. \textbf{445} (2017), no. 2, 1516--1529.


\bibitem{YAM} T.Yamazaki, \textit{On upper and lower bounds of the numerical radius and an equality condition,} Studia Math. \textbf{178} (2007), 83--89.

\end{thebibliography}

\end{document}